\documentclass[twoside,reqno]{amsart}

\usepackage[utf8]{inputenc}
\usepackage[T1]{fontenc}

\usepackage{amsmath,color,amsthm}
\usepackage{url}

\makeatletter
\@namedef{subjclassname@2020}{%
  \textup{2020} Mathematics Subject Classification}
\makeatother

\numberwithin{equation}{section}

\usepackage{xspace} 

\newcommand{\Lc}{{\mathcal L}}
\newcommand{\Lto}{{\stackrel{\Lc}{\longrightarrow}}}
\newcommand{\Leq}{{\,\stackrel{\Lc}{=}\,}}

\newcommand{\disteq}{\Leq}
\newcommand{\E}{\mathbb{E}}
\newcommand{\Prob}{\mathbb{P}}

\newcommand{\R}{\mathbb{R}}

\newcommand{\truth}{{\bf 1}}

\newcommand{\quickselect}{{\tt QuickSelect}\xspace}

\newcommand{\Li}{\text{Li}}

\newcommand{\quickmin}{{\tt QuickMin}\xspace}

\newcommand{\quicksort}{{\tt QuickSort}\xspace}
\newcommand{\quickval}{{\tt QuickVal}\xspace}
\newcommand{\quickquant}{{\tt QuickQuant}\xspace}

\newtheorem{theorem}{Theorem}[section]
\newtheorem{lemma}[theorem]{Lemma}
\newtheorem{proposition}[theorem]{Proposition}
\newtheorem{corollary}[theorem]{Corollary}
\newtheorem{Remark}[theorem]{Remark}
\newtheorem{Definition}[theorem]{Definition}

\providecommand{\norm}[1] {\left\lVert #1 \right\rVert}
\providecommand{\abs}[1] {\left\lvert #1 \right\rvert}
\providecommand{\eqref}[1] {(\ref{#1})}
\providecommand{\refL}[1] {\text{Lemma~}\ref{#1}}
\providecommand{\refT}[1] {\text{Theorem~}\ref{#1}}
\providecommand{\refP}[1] {\text{Proposition~}\ref{#1}}
\providecommand{\refC}[1] {\text{Corollary~}\ref{#1}}
\providecommand{\refS}[1] {\text{Section~}\ref{#1}}

\providecommand{\refR}[1] {\text{Remark~}\ref{#1}}

\newcommand{\Sns}{S_{n,\theta}}
\newcommand{\ts} {{\tau_\theta}}

\newcommand{\LOneTo}{\overset{L^1}{\longrightarrow}}

\newcommand{\ignore}[1]{}

\begin{document}

\title[Convergence of \quickval Residual]
{Convergence of the \quickval Residual} 

\newcommand\urladdrx[1]{{\urladdr{\def~{{\tiny$\sim$}}#1}}}
\author{James Allen Fill}
\address{Department of Applied Mathematics and Statistics,
The Johns Hopkins University,
34th and 
Charles Streets,
Baltimore, MD 21218-2682 USA}
\email{jimfill@jhu.edu}
\urladdrx{http://www.ams.jhu.edu/~fill/}
\thanks{Research for both authors was supported by the Acheson~J.~Duncan Fund for the Advancement of Research in Statistics.  The funder had no role in study design, decision to publish, or preparation of the manuscript.  The research of the second author was conducted while he was affiliated with The Johns Hopkins University.}

\author{Jason Matterer}
\address{Systems \& Technology Research LLC,
600 West Cummings Park, 
Woburn, MA 01801 USA}
\email{jason.matterer@stresearch.com}

\date{May~28, 2025}



\maketitle

\begin{center}
{\sc Abstract}
\vspace{.3cm}
\end{center}

\begin{small}
\quickselect (also known as {\tt Find}), introduced by Hoare \cite{h1961}, is 
a randomized algorithm for selecting a specified order statistic from an input sequence of~$n$ objects, or rather their identifying labels usually known as \emph{keys}.  The keys can be numeric or symbol strings, or indeed any labels drawn from a given linearly ordered set.  We discuss various ways in which the cost of comparing two keys can be measured, and we can measure the efficiency of the algorithm by the total cost of such comparisons.

We define and discuss a closely related algorithm known as \quickval and a natural probabilistic model for the input to this algorithm; \quickval searches (almost surely unsuccessfully) for a specified \emph{population} quantile 
$\alpha \in [0, 1]$ in an input sample of size~$n$.  Call the total cost of comparisons for this algorithm $S_n$.  We discuss a natural way to define the random variables $S_1, S_2, \ldots$ on a common probability space.  For a general class of cost functions, Fill and Nakama~\cite{fn2013} proved under mild assumptions that the scaled cost 
$S_n / n$
of \quickval converges in $L^p$ and almost surely to a limit random variable $S$.  For a general cost function, we consider what we term the \quickval residual:
\begin{equation*}
  \rho_n := \frac{S_n}n - S.
\end{equation*}
The residual is of natural interest, especially in light of the previous analogous work on the sorting algorithm 
{\tt QuickSort} \cite{bf2012,n2013,f2015,gk2016,s2017}.
In the case $\alpha = 0$ of {\tt QuickMin} with unit cost per key-comparison, we are able to 
calculate---\`{a} la Bindjeme and Fill for {\tt QuickSort}~\cite{bf2012}---the exact (and asymptotic) $L^2$-norm of the residual.  
We take the result as motivation for the scaling factor $\sqrt{n}$ for the {\tt QuickVal} residual for \emph{general} population quantiles and for \emph{general} cost.  We then prove \emph{in general} (under mild conditions on the cost function) that 
$\sqrt{n}\,\rho_n$
converges in law to a scale-mixture of centered Gaussians, 
and we also prove convergence of moments.  
\end{small}


\section{Introduction}\label{S:intro}

Parts of Sections \ref{S:intro}--\ref{S:set-up} are repeated nearly verbatim, for the convenience of the reader and with permission of the publisher, from \cite{fm2014}.  Note, however, that we have updated the literature review, perhaps most notably including an excellent sequel to this paper, namely, \cite{in2024}; see especially \refR{R:in2024}. 

In this section, we describe \quickselect and \quickquant and give historical background.  The main result of this paper, \refT{t:quickval_residual_limit}, concerns an algorithm very closely related to \quickquant known as \quickval, which is described in \refS{s:quickval_prelim}.  In \refS{c:exact_L2} we will define and consider the algorithm \quickmin, which can be viewed a special case of either \quickquant or \quickval. 

\quickselect (also 
known as {\tt FIND}), introduced by Hoare \cite{h1961}, is a randomized algorithm (a close cousin of the randomized sorting algorithm \quicksort, also introduced by Hoare \cite{h1962}) for selecting a specified order statistic from an input sequence of objects, or rather their identifying labels usually known as \emph{keys}.  The keys can be numeric or symbol strings, or indeed any labels drawn from a given linearly ordered set.  Suppose we are given keys $y_1,\ldots, y_n$ and we want to find the $m$th smallest among them. 
The algorithm first selects a key (called the pivot) uniformly at random. It then compares every other key to the pivot, thereby determining the rank, call it $r$, of the pivot among the~$n$ keys.  If $r = m$, then the algorithm terminates, returning the pivot key as output.  If $r > m$, then the algorithm is applied recursively to the keys smaller than the pivot to find the $m$th smallest among those; while if $r < m$, then the algorithm is applied recursively to the keys larger than the pivot to find the 
$(m - r)$th smallest among those.  More formal descriptions of \quickselect can be found in \cite{h1961} and \cite{k1972}, for example.

The cost of running \quickselect can be measured 
(somewhat crudely) 
by assessing the cost of comparing keys.  We assume that every comparison of two (distinct) keys costs some amount that is perhaps dependent on the values of the keys, and then the cost of the algorithm is the sum of the comparison costs.   

Historically, it was
customary to assign unit cost to each comparison of two keys, irrespective of their values.  We denote the (random) key-comparisons-count cost for {\tt QuickSelect} by $K_{n, m}$.
There have been many studies of the random variables $K_{n, m}$, including \cite{d1984}, \cite{mms1995}, \cite{gr1996}, \cite{MR1454110}, \cite{g1998}, \cite{d2001}, \cite{ht2002}, \cite{df2010}, \cite{fh2010}, and
\cite{fh2023}.
But unit cost is not always a reasonable model for comparing two keys.  For example, if each key is a string of symbols, then a more realistic model for the cost of comparing two keys is the value of the first index at which the two symbol strings differ.  To date, only a few papers (\cite{vcff2009}, \cite{fn2009}, \cite{fn2013}, and \cite{in2024}) have 
considered \quickselect from this more realistic symbol-comparisons perspective.  As in~\cite{fn2013} and~\cite{fm2014}, in this paper we will treat a rather general class of cost functions that includes both key-comparisons cost and symbol-comparisons cost.

In our set-up (to be described in detail in \refS{S:set-up}) for this paper, we will consider a variety of probabilistic models (called \emph{probabilistic sources}) for how a key is generated as an infinite-length string of symbols, but we will always assume that the keys form an infinite sequence of independent and identically distributed and almost surely distinct symbol strings.  This gives us, on a single probability space, all the randomness needed to run {\tt QuickSelect} for \emph{every} value of~$n$ and \emph{every} value of $m \in \{1, \dots, n\}$ by always choosing the \emph{first} key in the sequence as the pivot (and maintaining initial relative order of keys when the algorithm is applied recursively); this is what is meant by the \emph{natural coupling} 
(cf.~\cite[Section~1]{Fil2013}) of the runs of the algorithm for varying~$n$ and~$m$ (and varying cost functions).

When considering asymptotics of the cost of \quickselect as the number of keys tends to $\infty$, it becomes necessary to let the order statistic $m_n$ depend on the number of keys~$n$.  When 
$m_n / n \to \alpha \in [0, 1]$, we refer to \quickselect for finding the $m_n$th order statistic among~$n$ keys as {\tt QuickQuant}$(n, \alpha)$.  As explained in~\cite[Section~1]{Fil2013}, the natural coupling allows us to consider stronger forms of convergence for the cost of {\tt QuickQuant}$(n, \alpha)$ than convergence in distribution, such as almost sure convergence and convergence in $L^p$.
Fill and Nakama~\cite{fn2013} prove, under certain ``tameness'' conditions (to be reviewed later) on the probabilistic source and the cost function, that, for each fixed~$\alpha$, the cost of {\tt QuickQuant}$(n, \alpha)$, when scaled by~$n$, 
converges both in $L^p$ and almost surely to a limiting random variable.  Fill and Matterer~\cite{fm2014} extend these univariate convergence results to results about convergence of certain related stochastic processes.

Closely related to {\tt QuickQuant}$(n, \alpha)$ is an algorithm called {\tt QuickVal}$(n, \alpha)$, detailed in \refS{s:quickval_prelim}.  Employing the natural coupling, {\tt QuickVal}$(n, \alpha)$ searches (almost surely unsuccessfully) for a specified \emph{population} quantile $\alpha \in [0, 1]$ in an input sample of size~$n$.  Call the total cost of comparisons for this algorithm $S_n$.  For a general class of cost functions, Fill and Nakama~\cite{fn2013} proved under mild assumptions that the scaled cost $S_n / n$
of \quickval converges in $L^p$ and almost surely to a limit random variable $S$.  For a general cost function, we consider what we term the \quickval residual:
\begin{equation}
\label{rhon}
\rho_n := \frac{S_n}{n} - S.
\end{equation}
The residual is of natural interest, especially in light of the previous analogous work on 
{\tt QuickSort} \cite{bf2012,n2013,f2015,gk2016,s2017}.

An outline for this paper is as follows.  First, in \refS{S:set-up} we carefully describe our set-up and, in some detail, discuss probabilistic sources, cost functions, and tameness; we also discuss the idea of \emph{seeds}, which allow us a unified treatment of all sources. 
\refS{c:exact_L2} concerns {\tt QuickMin} (the case $\alpha = 0$ of \quickval) with unit cost per key-comparison, for which we are able to calculate---\`{a} la Bindjeme and Fill for {\tt QuickSort} \cite{bf2012}---the exact (and asymptotic) $L^2$-norm of the residual; the result is \refT{t:exact_L2}, which we take as motivation for the scaling factor $\sqrt{n}$ for the {\tt QuickVal} residual for \emph{general} population quantiles and for \emph{general} cost.  The remainder of the paper is devoted to establishing convergence of the \quickval\ residual.  \refS{s:quickval_prelim} introduces notation needed to state the main theorem and establishes an important preliminary result (\refL{l:sigma_limit}).  We state and prove the main theorem (\refT{t:quickval_residual_limit}), which asserts that the scaled cost of the \quickval\ residual converges in law to a scale mixture of centered Gaussians, in \refS{s:quickval_clt}; and in \refS{s:quickval_moments} we prove the corresponding convergence of moments.

\begin{Remark}
\emph{
As recalled from~\cite{fn2013} at the end of our \refS{S:sources}, many common sources, including memoryless and Markov sources, have the property that the source-specific cost function~$\beta$ corresponding to the symbol-comparisons cost for comparing keys is $\epsilon$-tame for every $\epsilon > 0$.  Thus our main result, \refT{t:quickval_residual_limit}, applies to all such sources.
}
\end{Remark}

\begin{Remark}
\label{R:in2024}
\emph{
In very recent work, Ischebeck and Neininger~\cite{in2024} extend our main \refT{t:quickval_residual_limit} from univariate normal convergence for each~$\alpha$ to Gaussian-process convergence, treating~$\alpha$ as a parameter, in the metric space of c\`{a}dl\`{a}g functions endowed with the Skorokhod metric.
}
\end{Remark}

To motivate the reader, here is a fairly easily understood instance of our main \refT{t:quickval_residual_limit}.  Suppose that keys arrive as i.i.d.\ uniform$(0, 1)$ random variables, and suppose that cost is measured classically as the number of key comparisons.  Using \quickval to search for population quantile $\alpha \in [0, 1]$, suppose for each $k \geq 0$ that the search has been narrowed to the (random) interval  $(L_k, R_k)$ after~$k$ steps of the algorithm have been carried out; in particular, $L_0 = 0$ and $R_0 = 1$.  Let $I_k := R_k - L_k$.  Then, as shown in the proof of \refT{t:quickval_residual_limit}, the random series (of positive terms) in the expressions
\begin{align*}
\sigma_{\infty}^2
&= \sum_{k = 1}^{\infty} (I_{k - 1} - I_{k - 1}^2) + 2 \sum_{\ell = 1}^{\infty} \sum_{k = 1}^{\ell - 1} (I_{\ell - 1} - I_{k - 1} I_{\ell - 1}) 
\\  
&= \sum_{k = 1}^{\infty} I_k (1 -I_k) + 2 \sum_{\ell = 2}^{\infty} I_{\ell} \sum_{k = 1}^{\ell - 1} (1 - I_k)
\end{align*}   
converge with probability one, and the residual $\rho_n$ given by~\eqref{rhon} converges in distribution to $\sigma_{\infty} Z$, where~$Z$ has a standard normal distribution and is independent of $\sigma_\infty$.

\section{Set-up}\label{S:set-up}

\subsection{Probabilistic sources}\label{S:sources}

Let 
us define the fundamental probabilistic structure underlying the analysis of {\tt QuickSelect}. We assume that keys arrive independently and with the same distribution and that each key is composed of a sequence of symbols from some
finite or countably infinite alphabet. Let $\Sigma$ be this alphabet (which we assume is totally ordered by $\leq$). Then a key is an element of $\Sigma^\infty$ [ordered by the lexicographic order, call it $\preceq$, corresponding to $(\Sigma,\leq)$] and a \emph{probabilistic source} is a stochastic process $W=(W_1, W_2, W_3, \ldots)$ such that for each $i$ the random variable $W_i$ takes values in $\Sigma$.  
We will impose restrictions on the distribution of $W$ that will have as a consequence that (with probability one) all keys are distinct.

We denote the cost (assumed to be nonnegative) of comparing two keys $w,w^\prime$ by ${\rm cost}(w, w^\prime)$.  As two examples, the choice 
${\rm cost}(w, w') \equiv 1$ gives rise to a key-comparisons analysis, whereas if words are symbol strings then a symbol-comparisons analysis is obtained by letting 
${\rm cost}(w, w')$ be the first index at which~$w$ and $w'$ disagree.

Since $\Sigma^\infty$ is totally ordered, a probabilistic source $W$ is governed by a distribution function $F$ defined for $w \in \Sigma^\infty$ by
$$
F(w) := \Prob (W \preceq w).
$$
Then the corresponding inverse probability transform~$M$, defined by 
 $$
 M(u):= \inf \left\{ w \in \Sigma^\infty : u \leq F(w) \right\},
 $$
has the property that if $U \sim \text{uniform}(0, 1)$, then $M(U)$ has the same distribution as $W$.  We refer to such uniform random variables~$U$ as \emph{seeds}.

Using this technique we can define a \emph{source-specific cost function} 
$$
\beta :(0,1) \times (0, 1) \rightarrow [0, \infty)
$$ 
by $\beta(u,v) := {\rm cost}(M(u), M(v))$.

\begin{Definition}
\label{d:tameness}
\emph{
Let $0< c < \infty$ and $0<\epsilon<\infty$.  A source-specific cost function~$\beta$ is said to be 
\emph{$(c,\epsilon)$-tame} if for $0 < u < t < 1$ we have
$$
\beta(u,t) \leq c\,(t-u)^{-\epsilon},
$$
and is said to be \emph{$\epsilon$-tame} if it is $(c, \epsilon)$-tame for some~$c$.
}
\end{Definition}

For further important background on sources, cost functions, and tameness, we refer the reader to Section~2.1 (see especially Definitions 2.3--2.4 and Remark~2.5) in Fill and Nakama~\cite{fn2013}.  Note in particular that many common sources, including memoryless and Markov sources, have the property that the source-specific cost function~$\beta$ corresponding to symbol-comparisons cost for comparing keys is $\epsilon$-tame for every $\epsilon > 0$.

\subsection{Tree of seeds and the \quickselect tree processes}\label{S:qsel_processes}

Let $\mathcal{T}$ be 
the collection of (finite or infinite) rooted ordered binary trees (whenever we refer to a binary tree we will assume it is of this variety) and let $\overline{T}\in\mathcal{T}$ be the complete infinite binary tree. We will label each node~$\theta$ in 
a given tree $T \in \mathcal{T}$
by a binary sequence representing the path from the root to $\theta$, where~$0$ corresponds to taking the left child and~$1$ to taking the right. We consider the set of  real-valued stochastic processes each with index set equal to some $T \in \mathcal{T}$.  For such a process, we extend the index set to $\overline{T}$ by defining $X_{\theta} = 0$ for $\theta \in \overline{T} \setminus T$.  
We will have need for the following definition of levels of a binary tree.

\begin{Definition}
\emph{
For $0\leq k <\infty$, we define the \emph{$k^{\rm th}$ level} $\Lambda_k$ of a binary tree as the collection of vertices that are at distance~$k$ from the root.
}
\end{Definition}

Let $\Theta = \bigcup_{0 \le k < \infty}\{0,1\}^k$ be the set of all finite-length binary strings, where $\{0,1\}^0 = \{\varepsilon\}$ with~$\varepsilon$ denoting the empty string. Set $L_\varepsilon := 0$, $R_\varepsilon := 1$, and $\tau_{\varepsilon} := 1$. Then, for $\theta \in \Theta$, we define $|\theta|$ to be the length of the string $\theta$, and $\upsilon_\theta(n)$ to be the size (through the arrival of the $n^{\text{th}}$ key) of the subtree rooted at node $\theta$. Given a sequence of independent and identically distributed (iid) seeds $U_1,U_2,U_3,\ldots$, we recursively define
\begin{eqnarray*}
\tau_\theta &:=& \inf\{i:\, L_\theta < U_i < R_\theta\}, \\
L_{\theta 0} &:=& L_\theta,\ L_{\theta 1} := U_{\tau_\theta}, \\
R_{\theta 0} &:=& U_{\tau_\theta},\ R_{\theta 1} := R_\theta,
\end{eqnarray*}
where $\theta_1 \theta_2$ denotes the concatenation of $\theta_1, \theta_2 \in \Theta$. 
For a source-specific cost function $\beta$ and $0\leq p <\infty$ we define 
\begin{align*}
S_{n,\theta} & := \sum_{\tau_\theta < i \leq n}\! \truth(L_\theta < U_i < R_\theta) \beta(U_i, U_{\tau_\theta}), \\
I_p(x,a,b) & := \int_{a}^{b}\! \beta^p(u,x)\,du, \\
I_{p,\theta} & := I_p(U_{\tau_\theta}, L_\theta, R_\theta), \\
I_\theta &:= I_{1,\theta}, \\
C_\theta &:= (\tau_\theta, U_{\tau_\theta}, L_\theta, R_\theta).
\end{align*}
In some later definitions we will make use of the positive part function defined as usual by 
$x^+ := x \truth(x > 0)$.
Given a source-specific cost function $\beta$ and the seeds $U_1,U_2,U_3,\ldots$, we define the $n$th 
{\tt QuickSelect} seed process 
as the $n$-nodes binary tree indexed stochastic process obtained by successive insertions of $U_1,\ldots, U_n$ into an initially empty binary search tree.

Before we use these random variables, we supply some understanding of them for the reader. 
The arrival time 
$\tau_\theta$ is the index of the seed that is slotted into node $\theta$ in the construction of the {\tt QuickSelect} seed process. Note that for each $\theta\in\Theta$ we have $P(\tau_\theta < \infty) = 1$. The interval $(L_\theta, R_\theta)$ provides sharp bounds for all seeds arriving after time $\tau_\theta$ that interact with $U_{\tau_\theta}$ in the sense of being placed in the subtree rooted at $U_{\tau_{\theta}}$. 
A crucial observation is that, conditioned on $C_\theta$, the sequence of seeds $U_{\tau_\theta+1}, U_{\tau_\theta + 2}, \ldots$ are iid uniform$(0,1)$; thus, again conditioned on $C_\theta$, the sum 
$S_{n,\theta}$ is the sum of $(n-\tau_\theta)^+$ iid random variables.  Note that when $n \leq \ts$ the sum defining $S_{n,\theta}$ is empty and so 
$S_{n,\theta}=0$; in this case we shall conveniently interpret $\Sns / (n-\ts)^+ =  0 / 0$ as $0$.
The random variable $S_{n, \theta}$ is the total cost of comparing the key with seed $U_{\tau_{\theta}}$ with keys (among the first~$n$
overall 
to arrive) whose seeds fall in the interval $(L_{\theta}, R_{\theta})$, and 
$I_{p, \theta}$ is the conditional $p$th moment of 
the cost of
one such comparison: 
If we let $U \sim\text{uniform}(0,1)$ independent of $C_\theta$, then
$$
I_{p,\theta} = \E \left[\left.\truth(L_\theta < U < R_\theta) 
\beta^p(U, U_{\tau_\theta})\right|C_\theta\right].
$$
Conditioned on $C_\theta$, the term $S_{n,\theta}$ is the sum of $(n-\tau_\theta)^+$ iid random variables with $p$th moment $I_{p,\theta}$.

We define the \emph{$n^{\rm th}$ {\tt QuickSelect} tree process} as the binary-tree-indexed stochastic process
$S_n = (\Sns)_{\theta\in\Theta}$ and the \emph{limit {\tt QuickSelect} tree process} 
(so called in light of \cite[Master Theorem~4.1]{fm2014})
by $I = (I_\theta)_{\theta\in\Theta}$.

We recall from~\cite{fm2014} an easily established lemma that will be invoked in~\refR{r:tameness} and in the proof of~\refL{l:QRMW}.
\begin{lemma}[Lemma~3.1 of~\cite{fm2014}]
\label{Ink}
If $\beta$ is $(c,\epsilon)$-tame with $0\leq\epsilon<1/s$, then for each fixed node $\theta\in \Lambda_k$ and $0\leq r < \infty$ we have
$$
\E I_{s,\theta}^r \leq \left(\frac{2^{s\epsilon}c^s}{1-s\epsilon}\right)^r\left(\frac{1}{r+1-rs\epsilon}\right)^k.~\qed
$$
\end{lemma}


\section{Exact $L^2$ Asymptotics for \quickmin Residual}\label{c:exact_L2}

Before deriving a limit law for \quickval under general source-specific cost functions $\beta$, we motivate the scaling factor of 
$\sqrt{n}$ in \refT{t:quickval_residual_limit}. 
We consider the case of \quickmin (i.e.,\ \quickselect for the minimum key) with key-comparisons cost ($\beta \equiv 1$). 
Note that the operation of \quickmin and \quickval with $\alpha=0$ are identical. 
Our goal In this section is to establish \refT{t:exact_L2}, which gives exact and asymptotic expansions for the second moment of the residual in this special case.

Let $K_n$ denote the key-comparisons cost of \quickmin and define 
\begin{equation}\label{eq:Y_ndef}
Y_n := \frac{K_n - \mu_n}{n+1},
\end{equation}
where $\mu_n :=\ \E K_n = 2(n-H_n)$ for each $n$ \cite{k1972}. 
A consequence of \cite[Theorem~1]{ht2002} is that $K_n/n \Lto D$, where $D\disteq \sum_{k=0}^{\infty} \prod_{j=0}^k U_j$ has a Dickman distribution \cite{ht2002}, with $\mu := \E D = 2$ (here $U_0 := 1$).
(Note that \cite{ht2002} refers to $D-1$ as having a Dickman distribution; we ignore this distinction.)
Applying \cite[Theorems~3.1 and~3.2]{fn2013} to 
the special case of \quickmin using key-comparisons costs yields the stronger result that $Y_n$ converges to a limit random variable~$Y$ in $L^p$ for any $p\geq 1$ and almost surely.  We can then set
\begin{equation}\label{eq:Y_D1}
  D := Y + \mu = Y + 2,
\end{equation}
and this~$D$ has a Dickman distribution as defined above.

The main result of this \refS{c:exact_L2} is the exact calculation of of the second moment of $Y_n -Y$:

\begin{theorem}\label{t:exact_L2}
For $Y_n$ and $Y$ defined previously, we have
\begin{align*}
a_n^2 &:= \E(Y_n - Y)^2 =  (n + 1)^{-2} \left[ \frac{3}{2}n + 4H_n  - 4 H_n^{(2)} + \frac{1}{2} \right]\\
&= \frac{3}{2} n^{-1} + O\left(\frac{\log n}{n^2}\right).
\end{align*}
\end{theorem}

The remainder of this section builds to the proof of~\refT{t:exact_L2}.
Define 
\begin{equation}\label{eq:I_ndef}
N_n:= \#\{1<i\leq n : U_i < U_1\}
\end{equation}
(i.e.,\ the number of keys that fall into the left subtree of the \quickselect seed process).
To begin the derivation, note that $K_n = n - 1 + \widetilde{K}_{N_n}$, where $\widetilde{K}_{N_n}$ is the key-comparisons cost for \quickmin applied to the left subtree.  
Note also that the same equation holds as equality in law if the process~$\widetilde{K}$ has the same distribution as the process~$K$ and is independent of $N_n$. 
We also have $D = 1 + U \widetilde{D}$ with $U := U_1$  and $\widetilde{D} \disteq D$ independent.

Make the following definitions:
\begin{align}
Y_{n,0} &:= \frac{K_{N_n}-\mu_{N_n}}{N_n+1}, \nonumber\\
Y^{(0)} &:= \widetilde{D}-\mu.\label{eq:Y_D2}
\end{align}
Then we can express the residual $Y_n-Y$ in terms of these ``smaller versions'' of 
$Y_n$ and $Y$:
\begin{align}
	Y_n - Y &= \frac{n-1 + K_{N_n} - \mu_n}{n+1} - \left( 1 +\ U \widetilde{D} -\mu \right)\nonumber\\
	&= \left(\frac{N_n+1}{n+1}\right)Y_{n,0} - UY^{(0)} + \frac{n-1}{n+1} - 1 + \frac{\mu_{N_n} - \mu_n}{n+1} - U\mu + \mu \nonumber\\
	&=\left(\frac{N_n+1}{n+1}\right)Y_{n,0} - UY^{(0)} + C_n(N_n)\frac{n}{n+1} - C(U),\label{eq:Y_n_recurrence}
\end{align}
where $C_n(i):= n^{-1}(n-1 + \mu_i-\mu_n)$ and $C(x):= \mu x - 1 = 2x-1$. Observe that with these definitions, we can break up the previous equation as
\begin{equation}\label{eq:Y_W_split}
	Y_n - Y = W_1 + W_2,
\end{equation}
where 
\[
W_1 := \frac{N_n+1}{n+1} Y_{n,0} - UY^{(0)}, \quad W_2:=C_n(N_n) \frac{n}{n+1} - C(U).
\]
 Conditionally given $N_n$ and $U$, the random variable $W_2$ is constant and $W_1$ has mean zero, so 
$$
a_n^2=\E\left(Y_n -Y\right)^2 = \E W_1^2 + \E W_2^2.
$$

Consider the first term $\E W_1^2$.

\begin{lemma}

 \[
  \E W_1^2 = \frac{1}{n} \sum_{k=0}^{n-1} \frac{(k+1)^2}{(n+1)^2} a_k^2 + \frac{1}{12(n+1)}.
 \]

\end{lemma}
\begin{proof}
If we define 
\[
	Z_1 := \frac{N_n+1}{n+1} \left(Y_{n,0}-Y^{(0)}\right), \quad 
	Z_2 := \left(\frac{N_n+1}{n+1} - U\right) Y^{(0)},
\]
then $W_1 = Z_1 + Z_2$ and  so $\E W_1^2 = \E Z_1^2+\E Z_2^2 + 2\E (Z_1 Z_2)$.

For the cross term $\E (Z_1 Z_2)$, conditionally given $N_n$ the random variable $U$ is distributed Beta$(N_n+1,n-N_n)$. Therefore,
$$
\E (Z_1 Z_2)
= \E\left\{\E \left[\left.Z_1\left(\frac{N_n+1}{n+1} - U\right)Y^{(0)}\right|N_n, Y_{n,0},Y^{(0)}\right]\right\} = 0.
$$

Next consider the term $\E Z_1^2$.  
\begin{Remark}\label{r:Y*}
\rm
The conditional joint distribution of the process $(Y_{n, 0})_{n \geq 0}$ and the random variable $Y^{(0)}$ given $N_n$ is the conditional joint distribution of the process $(Y^*_{N_n})_{n \geq 0}$ and the random variable $Y^*$ given $N_n$, where the process 
$(Y^*_n)$ and the random variable $Y^*$ are independent of $N_n$ and have (unconditionally) the same joint distribution as the process $(Y_n)$ and the random variable $Y$.
\end{Remark}
In light of the preceding remark,
$$
\E Z_1^2 = \E\left[\left(\frac{N_n+1}{n+1}\right)^2\left(Y_{N_n}^* - Y^*\right)^2\right].
$$
Since $N_n\sim \text{unif}\left\{0,1,2,\ldots,n-1\right\}$, conditioning on $N_n$ gives
$$
\E Z_1^2 = \frac{1}{n} \sum_{k=0}^{n-1} \left(\frac{k+1}{n+1}\right)^2 a_k^2.
$$

Finally, consider the term $\E Z_2^2$.  Since $Y^{(0)}$ is independent of $N_n$ and $U$, we have
$$
\E Z_2^2 = {\E Y^{(0)}}^2\,\E \left(\frac{N_n+1}{n+1} - U\right)^2.
$$
Recall that $N_n \sim \text{unif}\left\{0,1,\ldots,n-1\right\}$ and that conditionally given $N_n$, we have $U\sim \text{Beta}(N_n+1,n-N_n)$; therefore,
\begin{equation}\label{eq:Z2_exp}
\E\left(\frac{N_n+1}{n+1} - U\right)^2 = \frac{1}{n} \sum_{k=0}^{n-1} \frac{(k+1)(n-k)}{(n+1)^2(n+2)} = \frac{1}{6(n+1)}.
\end{equation}
Since ${\E Y^{(0)}}^2 = 1/2$ \cite{kp1998}, 
we have that 
$$
\E Z_2^2 = \frac{1}{12(n+1)}.
$$
Putting these calculations together, we get that
$$
\E W_1^2 = \frac{1}{n} \sum_{k=0}^{n-1} \frac{(k+1)^2}{(n+1)^2} a_k^2 + \frac{1}{12(n+1)}.
$$
\end{proof}

Now we consider the term $\E W_2^2$.
\begin{lemma}
\label{L:EW_2^2}
 \[
  \E W_2^2 = \frac{2}{3(n+1)} + \left(\frac{2}{n+1}\right)^2\left(1-\frac{H_n}{n}\right).
 \]

\end{lemma}
\begin{proof}
We have
\begin{align}
W_2 &= C_n(N_n)\frac{n}{n+1} - C(U) \nonumber \\
&= \frac{n}{n+1}\left[\frac{1}{n}\left(n-1 + 2(N_n-H_{N_n})-2(n-H_n)\right)\right] + 1-2U \nonumber\\
&= \frac{1}{n+1}\left[2(N_n-H_{N_n}) - n +2H_n - 1\right] + 1 - 2U \nonumber \\
&= 2\left(\frac{N_n+1}{n+1} - U\right) - \frac{2}{n+1} + \frac{1}{n+1}\left(-n+2H_n -2H_{N_n} - 1\right) + 1 \nonumber \\
&= 2\left(\frac{N_n+1}{n+1} - U\right) + \frac{2}{n+1}\left(H_n-H_{N_n} - 1\right). \label{eq:C_remainder}
\end{align}
Squaring and then taking expectations, we find

\begin{align*}
	\E W_2^2 =& 4\E\left(\frac{N_n+1}{n+1} - U\right)^2 + \E \left[\left(\frac{2}{n+1}\right)^2\left(H_n-H_{N_n}-1\right)^2\right] \\
	&\quad+ 4\E\left[\left(\frac{N_n+1}{n+1} - U\right)\left(\frac{2}{n+1}\left(H_n-H_{N_n}-1\right)\right)\right].
\end{align*}
Recall that conditionally given $N_n$ we have $U\sim \text{Beta}(N_n+1,n-N_n)$, which implies that the cross term vanishes; therefore, it suffices to consider the two squared terms. From 
\eqref{eq:Z2_exp} we know that
$$
4\E\left(\frac{N_n+1}{n+1} - U\right)^2= \frac{2}{3(n+1)}.
$$
We now proceed to treat the final term
\begin{align}
L_n&:=\E\left[\left(\frac{2}{n+1}\right)^2\left(H_n-H_{N_n} - 1\right)^2\right] \nonumber \\
\label{eq:foil}
&= \left(\frac{2}{n+1}\right)^2 \left[ \E(H_n - H_{N_n})^2 - 2 \E(H_n - H_{N_n}) + 1 \right].
\end{align}
We can compute the first and second moments of $H_n - H_{N_n}$ as follows.  Fixing~$n$, for $1 \leq i \leq n$ consider the events $B_i := \{N_n < i\}$, which satisfy $\Prob(B_i) = \frac{i}{n}$ and $B_i \cap B_j = B_{i \wedge j}$ (where~$\wedge$ denotes minimum).  Note that
\[
H_n - H_{N_n} = \sum_{i = N_n + 1}^n \frac{1}{i} = \sum_{i = 1}^n \frac{\truth(B_i)}{i}
\]
Thus
\[
\E(H_n - H_{N_n}) = \sum_{i = 1}^n \frac{\Prob(B_i)}{i} = \sum_{i = 1}^n \frac{1}{n} = 1
\]
and
\begin{align*}
\E(H_n - H_{N_n})^2 
&= \E\left[ \sum_{i = 1}^n \frac{\truth(B_i)}{i} \right]^2 
= \sum_{i = 1}^n \sum_{j = 1}^n \frac{\Prob(B_i \cap B_j)}{i j}
= \frac{1}{n} \sum_{i = 1}^n \sum_{j = 1}^n \frac{i \wedge j}{i j} \\
&= \frac{2}{n} \sum_{i = 1}^n \sum_{j = 1}^i \frac{j}{i j} - \frac{1}{n} \sum_{i = 1}^n \frac{i}{i^2}
= 2 - \frac{H_n}{n}.
\end{align*}
Therefore we get that
$$
\E W_2^2 = \frac{2}{3(n+1)} + \left(\frac{2}{n+1}\right)^2 \left(1-\frac{H_n}{n}\right),
$$
as desired.
\end{proof}

We will also need the following well-known (and very easy to derive) solution to a ``divide-and-conquer'' recurrence 
in the proof of \refT{t:exact_L2}.

\begin{lemma}\label{l:rec_lemma}
 Let $(A_n)_{n\geq 0}$ and $(B_n)_{n\geq 1}$ be sequences of real numbers that satisfy
 \[
  A_n = \frac{1}{n} \sum_{k=0}^{n-1} A_k + B_n,
 \]
 for $n\geq 1$.
 Then for $n\geq 0$
 we have
 \begin{equation}\label{eq:ind_hyp}
  A_n = A_0 + B_n + \sum_{k=1}^{n-1} \frac{1}{k+1} B_k,
 \end{equation}
with
$B_0 := 0$.~\qed
\end{lemma}
\medskip

\begin{proof}[Proof of Theorem~\ref{t:exact_L2}.]
Combining the expressions for $\E W_1^2$ and $\E W_2^2$ gives
\begin{equation}\label{eq:recurrence}
(n+1)^2 a_n^2 = \frac{1}{n}\sum_{k=0}^{n-1} (k+1)^2a_k^2 + \frac{n+1}{12} + \frac{2(n+1)}{3} + 4\left(1-\frac{1}{n}H_n \right).
\end{equation}
If we define
$$
b_n:= \frac{n+1}{12} + \frac{2(n+1)}{3} + 4\left(1-\frac{1}{n}H_n \right) = \frac{3(n+1)}{4}
+ 4\left(1-\frac{1}{n}H_n \right),
$$
then 
\refL{l:rec_lemma} 
implies
$$
(n+1)^2a_n^2 = a_0^2 + \sum_{j=1}^{n-1} \frac{b_j}{j+1} + b_n.
$$
Plugging in $b_n$ gives
$$
(n+1)^2a_n^2 = a_0^2 + \sum_{j=1}^{n-1}\left[\frac{3(j+1)}{4(j+1)}+ \frac{4}{j+1}\left(1-\frac{1}{j}H_j \right)\right]  + \frac{3 (n + 1)}{4}+ 4\left(1-\frac{1}{n}H_n \right).
$$
Simplifying this expression gives
\begin{align*}
(n+1)^2a_n^2 
&= a_0^2 + \frac{3}{4}(n-1) + \sum_{j=1}^{n-1}\frac{4}{j+1}-\sum_{j=1}^{n-1}\frac{4}{j(j+1)}H_j 
+ \frac{19}{4}\\ 
&\qquad \ {} + \frac{3 n}{4} -\frac{4}{n}H_n\\
&= \frac{1}{2} + \frac{3}{2}n + 4H_n - \frac{4}{n} H_n -\sum_{j=1}^{n-1}\frac{4}{j(j+1)}H_j\\
&= \frac{1}{2} + \frac{3}{2}n + 4H_n - \frac{4}{n} H_n - 4 \left( H_n^{(2)} - \frac{H_n}{n} \right)\\
&= \frac{3}{2}n + 4H_n  - 4 H_n^{(2)} + \frac{1}{2},
\end{align*}
where $a_0^2=1/2$ was substituted in the second equality. Therefore, we can conclude that
\begin{align*}
a_n^2 
&= \E (Y_n - Y)^2 = (n + 1)^{-2} \left[ \frac{3}{2}n + 4H_n  - 4 H_n^{(2)} + \frac{1}{2} \right]\\
&= \frac{3}{2} n^{-1} + O\left(\frac{\log n}{n^2}\right).
\end{align*}

\end{proof}

\section{\quickval\ and mixing distribution for residual limit distribution}\label{s:quickval_prelim}

Our main theorem, \refT{t:quickval_residual_limit}, asserts that the scaled \quickval\ residual cost converges in law to a scale mixture of centered Gaussians.  In this section, we introduce needed notation and prove \refL{l:sigma_limit}, which gives an explicit representation of the mixing distribution. 

Consider a binary search tree (BST) constructed by the insertion (in order) of the $n$ seeds. Then 
{\tt QuickQuant}($n, \alpha$) follows the path from the root to the 
node storing the $m_n^{\rm th}$ smallest key, where $m_n/n\rightarrow \alpha$.

For {\tt QuickVal}($n,\alpha$), consider the same BST of seeds with the additional value~$\alpha$ inserted (last).  Then {\tt QuickVal}($n,\alpha$) follows the path from the root to this $\alpha$-node. Almost surely for $n$ large and~$k$ fixed, the difference between these two algorithms in costs associated with the $k$-th pivot is negligible to lead order \cite[(4.2)]{fn2013}. See \cite{vcff2009} or \cite{fn2013} for a more complete description.

When considering \quickval, 
we will simplify the notation since we will only need to reference one path of nodes from the root to a leaf in the \quickselect\ process tree. For this we define similar notation indexed by the pivot index (that is, by the level in the tree).
Set $L_0 := 0$, $R_0 := 1$, and $\tau_0 := 1$. Then, for $k \geq 1$, we define
\begin{eqnarray}
\tau_k &:=& \inf\{i:\, L_{k-1} < U_i < R_{k-1}\},\label{pivotIndex} \\
L_k &:=& {\bf 1}(U_{\tau_k} < \alpha) U_{\tau_{k}} + {\bf 1}(U_{\tau_k} > \alpha) L_{k-1},\label{lowerBound} \\
R_k &:=& {\bf 1}(U_{\tau_k} < \alpha) R_{k-1} + {\bf 1}(U_{\tau_k} > \alpha) U_{\tau_{k}},\label{upperBound} \\
C_k &:=& (L_{k-1}, R_{k-1}, \tau_k, U_{\tau_k}) \label{Ck}\\
X_{k,i} &:=& \truth(L_{k-1} < U_i < R_{k-1})\beta(U_i,U_{\tau_k}),\label{Xki} \\
S_{k,n} &:=& \sum_{i:\,\tau_k < i \leq n} X_{k,i}.\label{Skn}
\end{eqnarray}

\begin{Remark}
\emph{Note that~\cite{fn2013} used the notation $S_{n,k}$ for what we have called $S_{k,n}$. 
}
\end{Remark}

The random variable $\tau_k$ is the arrival time/index of the $k^{\rm th}$ pivot.
The interval $(L_k,R_k)$ gives the range of seeds to be compared to the $k^{\rm th}$ pivot in the operation of the \quickval\ algorithm.
The cost of comparing seed $i$ to the $k^{\rm th}$ pivot is given by $X_{k,i}$.
The total comparison costs attributed to the $k^{\rm th}$ pivot is $S_{k,n}$.

The cost of \quickval\ on~$n$ keys is then given by
\begin{equation}\label{eq:quickval_def}
S_n := \sum_{k=1}^\infty S_{k,n}.
\end{equation}

Define
\begin{equation*}
  \widehat{C}_K := \{C_k: k=1,\ldots ,K\},
\end{equation*}
and
\begin{equation*}
  \widehat{X}_{K,i}:= \sum_{k=1}^{K} X_{k,i}.
\end{equation*}
Then, conditionally given $\widehat{C}_K$, the random variable
\begin{equation*}
  \widehat{S}_{K,n}:= \sum_{\tau_K < i \leq n} \widehat{X}_{K,i}
\end{equation*}
is the sum of $(n-\tau_K)^+$ independent and identically distributed random variables, each with the same conditional distribution as $\widehat{X}_K := \sum_{k = 1}^K X_k$, where 
$$
X_k := {\bf 1}(L_{k - 1} < U < R_{k - 1}) \beta(U, U_{\tau_k})
$$
and~$U$ is uniformly distributed on 
$(0, 1)$ and independent of all the $U_j$'s. Here, $\widehat{X}_{K,i}$ is the cost incurred by comparing seed $i$ to pivots $1,2,\ldots K$ and $\widehat{S}_{K,n}$ is the comparison cost of all seeds that arrive after the $K$-th pivot to pivots $1,2,\ldots K$.

It will be helpful to condition on $\widehat{C}_K$ later. 
In anticipation of this, we establish notation for the conditional expectation of $X_k$ given $C_k$ (which equals the conditional expectation given $\widehat{C}_k$) and, for $k \leq \ell$, the conditional expected product of $X_k$ and $X_\ell$ given 
$\widehat{C}_{\ell}$, as follows:
\begin{align}\label{eq:I_k}
I_k :=& \E[X_k|C_k] = \int_{L_{k-1}}^{R_{k-1}} \beta(u,U_{\tau_k})\, du,  \\
I_{2,k,\ell} \label{eq:I_2kl}
:=& \E[X_k X_\ell|\widehat{C}_{\ell}]
= \int_{L_{\ell-1}}^{R_{\ell-1}} \beta(u,U_{\tau_k})\beta(u,U_{\tau_\ell})\, du.
\end{align}
We symmetrize the definition of $I_{2, k, \ell}$ in the indices~$k$ and~$\ell$ by setting
$I_{2, k, \ell}: = I_{2, \ell, k}$ for $k > \ell$.  Finally, we write $I_{2, k}$ as shorthand for $I_{2, k, k}$.

We now calculate the mean and variance of $\widehat{X}_K$ with the intention of applying the classical central limit theorem; everything is done conditionally given $\widehat{C}_K$.  Define 
$\mu_{K}$ and $\sigma_K^2$ to be the conditional mean and conditional variance of $\widehat{X}_K$ given $\widehat{C}_K$, respectively.  Then
\begin{equation*}
  \mu_K = \sum_{k=1}^K I_k, \quad\quad \sigma_K^2 =  \sum_{1\leq k, \ell \leq K} \left(I_{2,k,\ell} - I_kI_\ell\right).
\end{equation*}
We next present a condition guaranteeing that $\sigma_K^2$ behaves well as $K \to \infty$. 
We note in passing that this condition is also the sufficient condition of Theorem~3.1 in~\cite{fn2013} ensuring that $S_n / n$ converges in $L^2$ to
\begin{equation}
\label{eq:Sdef}
S := \sum_{k \geq 1} I_k.
\end{equation}

\begin{lemma}
\label{l:sigma_limit}
 If
  \begin{equation}
    \sum_{k=1}^\infty \left(\E I_{2,k}\right)^{1/2} < \infty,
    \label{eq:sigma_infty_assumption}
  \end{equation}
then both almost surely and in $L^1$ we have that {\rm (i)}~the two series on the right in the equation
  \begin{equation}
  \label{eq:sigma_infty}
    \sigma^2_\infty 
    := \sum_{k=1}^\infty (I_{2,k}-I_k^2) + 2\sum_{\ell=1}^\infty \sum_{k=1}^{\ell-1} (I_{2,k,\ell} - I_k I_\ell).
  \end{equation}
converge absolutely,  {\rm (ii)}~the equation holds, and  {\rm (iii)} $\sigma_K \LOneTo \sigma_{\infty}$ as 
$K \to \infty$.
\end{lemma}
\noindent

\begin{proof}
Recall the notation $X_k = {\bf 1}(L_{k - 1} < U < R_{k - 1}) \beta(U, U_{\tau_k})$ from above.  Consider $1 \leq k \leq \ell$.  The term  $I_{2,k,\ell} - I_k I_\ell$ equals the conditional covariance of $X_k$ and $X_{\ell}$ given $\widehat{C}_{\ell}$, and the absolute value of this conditional covariance is bounded above by the product of the conditional $L^2$-norms, namely, $I_{2, k}^{1/2} I_{2, \ell}^{1/2}$.  Thus for the three desired conclusions it is sufficient that $\E \left( \sum_{k = 1}^{\infty} I_{2, k}^{1/2} \right)^2 < \infty$.  But
$$
\E \left( \sum_{k = 1}^{\infty} I_{2, k}^{1/2} \right)^2 
= \left\| \sum_{k = 1}^{\infty} I_{2, k}^{1/2} \right\|_2^2
\leq \left( \sum_{k = 1}^{\infty} \left\| I_{2, k}^{1/2} \right\|_2 \right)^2
 = \left( \sum_{k = 1}^{\infty} \left(\E I_{2,k}\right)^{1/2} \right)^2.
$$
\end{proof}
\begin{Remark}
\emph{
In light of the absolute convergence noted in conclusion~(i) of Lemma \ref{l:sigma_limit}, we may unambiguously write
\begin{equation}
\label{eq:sigma_infty_simplified}
\sigma^2_{\infty} = \sum_{1\leq k, \ell < \infty} (I_{2, k, \ell} - I_k I_{\ell}),
\end{equation}
both in $L^1$ and almost surely.
}
\end{Remark}
\begin{Remark}
\label{r:tameness}
\emph{
Note that if the source-specific cost function $\beta$ is $\epsilon$-tame for some $\epsilon<1/2$, then,
by \refL{Ink} with $s = 2$ and $r = 1$,
condition~\eqref{eq:sigma_infty_assumption} in \refL{l:sigma_limit} is satisfied,
because the series there enjoys geometric convergence.
}
\end{Remark}

\section{Convergence}\label{s:quickval_clt}

Our main result is that, for a suitably tame cost function, the \quickval\ residual 
converges in law to a scale-mixture of centered Gaussians. 
Furthermore, we have the explicit representation of \refL{l:sigma_limit} for the random scale $\sigma_{\infty}$ as an infinite series of random variables that depend on conditional variances and covariances related to the source-specific cost functions [see~\eqref{eq:sigma_infty_simplified} and \eqref{eq:I_k}--\eqref{eq:I_2kl}].
\begin{theorem}\label{t:quickval_residual_limit}
  Suppose that the cost function $\beta$ is $\epsilon$-tame with $\epsilon < 1/2$.  Then
  \begin{equation*}
   \sqrt{n} \left( \frac{S_n}{n} - S \right)\,\Lto\,\sigma_\infty Z,
  \end{equation*}
  where $Z$ has a standard normal distribution and is independent of $\sigma_\infty$.
  \label{t:quickval}
\end{theorem}

We approach the proof of \refT{t:quickval_residual_limit} in two parts. First, in \refP{p:approx_limit} we apply the central limit theorem to an approximation $\widehat{S}_{K,n}$ of the cost of 
\quickval\ $S_n$.
Second, we show that the error due to the approximation $\widehat{S}_{K,n}$ is negligible in the limit, culminating in the results of Propositions~\ref{p:W_n_limit} and \ref{p:n_tau_k_limit}. 

Before proving \refT{t:quickval_residual_limit}, we state a corollary to \refT{t:quickval_residual_limit} for \quickmin. 
Recall that \quickmin\ is \quickselect\ applied to find the minimum of the keys.
Using a general source-specific cost function $\beta$, denote the cost of \quickmin\ on $n$~keys by $V_n$. 
Since the operation of \quickmin\ is the same as that of \quickval\ with $\alpha=0$, \refT{t:quickval} implies the following convergence for the cost of \quickmin\ with the same mild tameness condition on the source-specific cost function. 

\begin{corollary}
\label{c:quickmin}
 Suppose that the source-specific cost function $\beta$ is $\epsilon$-tame with $\epsilon < 1/2$.  Then
  \begin{equation*}
    \sqrt{n} \left( \frac{V_n}{n} - S \right)\,\Lto\,\sigma_\infty Z,
  \end{equation*}
  where $Z$ has a standard normal distribution and is independent of $\sigma_\infty$.
\end{corollary}

\begin{Remark}\label{r:sigma_infinity_quickmin}
{\em
In the key-comparisons case $\beta = 1$ (which is $\epsilon$-tame for every $\epsilon \geq 0$)
for $k \geq 0$ we have
$L_k \equiv 0$ and $R_k \equiv U_{\tau_k}$, with the convention $U_{\tau_0} := 1$.  Hence 
$I_k = U_{\tau_{k - 1}}$ for $k \geq 1$, and $I_{2, k, \ell} = U_{\tau_{\ell - 1}}$ for $1 \leq k \leq \ell$.
Therefore 
$S = \sum_{k \geq 1} U_{\tau_{k - 1}} = 1 + \sum_{k \geq 1} U_{\tau_k}$ and
$$
\sigma_{\infty}^2 
= \sum_{1 \leq k, \ell < \infty} (1 - U_{\tau_k}) U_{\tau_{\ell}}
$$
in~\refC{c:quickmin}.  To further simplify the understanding of 
$\sigma^2_{\infty}$, and hence of the limit
in \refC{c:quickmin} in this case, observe that $U_{\tau_1}, U_{\tau_2}, \dots$ have the same joint distribution as the cumulative products $U_1, U_1 U_2, \dots$.
Thus
\begin{equation*}
  \sigma_\infty^2 \Leq \sum_{1\leq k,\ell <\infty} \left[\left(1 - \prod_{i=1}^k U_i \right)\prod_{j=1}^\ell U_j\right].
\end{equation*}
}
\end{Remark}
\bigskip

Define
\begin{equation*}
  T_{K,n}:= \frac{\widehat{S}_{K,n} - (n - \tau_K)^+ \mu_K}{\sqrt{n}}.
\end{equation*}

\begin{proposition}\label{p:approx_limit}
Fix $K \in \{1, 2, \dots\}$.
  Suppose that
\begin{equation*}
    \E I_{2,k} < \infty
  \end{equation*}
for $k=1,2,\ldots, K$. Then
\begin{equation*}
  T_{K,n\,}\Lto\,\sigma_K Z
\end{equation*}
as $n\rightarrow\infty$, where $Z$ has a standard normal distribution independent of $\sigma_K$.
\end{proposition}
\begin{proof}

The 
classical central limit theorem for independent and identically distributed random variables 
applied conditionally given $\widehat{C}_K$ yields
\begin{equation}
  \Lc\left(\frac{\widehat{S}_{K,n} - (n - \tau_K)^+ \mu_K}{\sqrt{(n-\tau_K)^+}}\Big| \widehat{C}_K\right)\,\Lto\,\rm{N}(0,\sigma_K^2).
  \label{eq:cond_clt}
\end{equation}
Since $\tau_K$ is finite almost surely, Slutsky's theorem (applied conditionally given $\widehat{C}_K$) implies that we can replace $\sqrt{(n-\tau_K)^+}$ by $\sqrt{n}$ in the denominator of \eqref{eq:cond_clt}. 
Finally, applying the dominated convergence theorem to conditional distribution functions gives that the resulting conditional convergence in distribution in~\eqref{eq:cond_clt} holds unconditionally.
\end{proof}
Define
\begin{equation*}
  W_{K,n} := \frac{1}{\sqrt{n}} \sum_{k=1}^K \sum_{\tau_k < i \leq n} (X_{k,i} - I_k), \quad\quad
  \overline{W}_{K,n} := \frac{1}{\sqrt{n}} \sum_{k=K+1}^\infty \sum_{\tau_k < i \leq n} (X_{k,i} - I_k),
\end{equation*}
and let
\begin{equation*}
  W_{n} := W_{K,n} + \overline{W}_{K,n}.
\end{equation*}
Note that $W_n$ does not depend on $K$.  We can write $W_n$ in terms of the cost of \quickval\ as follows:
\begin{equation*}
W_n = \frac{1}{\sqrt{n}} \left( S_n - \sum_{k=1}^\infty (n-\tau_k)^+ I_k \right).
\end{equation*}
We prove that $W_n\,\Lto\,\sigma_\infty Z$ (which is \refP{p:W_n_limit}) in three parts. 
First (Lemma \ref{l:small_k}) we show that $\abs{T_{K,n} - W_{K,n}}\to 0$ almost surely.  
Next (\refL{l:large_k}) we show that $\norm{\overline{W}_{K,n}}_2$ is negligible as first $n \rightarrow\infty$ and then $K\rightarrow\infty$. 
Lastly (see the proof below of \refP{p:W_n_limit}), an application of Markov's inequality gives the desired convergence.

\begin{lemma}
  For $K$ fixed, if $\E I_k < \infty$ for $k=1,2,\ldots,K$, then
  \begin{equation*}
    \abs{T_{K,n} - W_{K,n}} \to 0 
  \end{equation*}
almost surely as $n\rightarrow\infty$. 
  \label{l:small_k}
\end{lemma}

\begin{Remark}
\emph{
  The condition $\E I_k < \infty$ in~\refL{l:small_k} is weaker than the condition $\E I_{2,k} < \infty$ in~\refP{p:approx_limit}.
}
\end{Remark}

\begin{proof}[Proof of \refL{l:small_k}]
When $n > \tau_K$ we have
  \begin{align*}
    \abs{T_{K,n}-W_{K,n}} &= \frac{1}{\sqrt{n}}\abs{\sum_{k=1}^K \sum_{\tau_k < i \leq \tau_K} (X_{k,i} - I_k)} \\
    &\leq \frac{1}{\sqrt{n}} \sum_{k=1}^K \sum_{\tau_k < i \leq \tau_K} \abs{X_{k,i} - I_k}.
  \end{align*}
  For a fixed $K$ with $k\leq K$, the almost sure finiteness of $\tau_k$ and $\tau_K$ implies that the sum 
  \begin{equation}
    \sum_{k=1}^K \sum_{\tau_k < i \leq \tau_K} \abs{X_{k,i} - I_k},
    \label{eq:small_k_sum}
  \end{equation}
  consists of an almost surely finite number of terms. 
  Since each term $\abs{X_{k,i} - I_k}$ is finite almost surely, the sum in \eqref{eq:small_k_sum} is finite almost surely.
  Therefore, $\abs{T_{K,n}-W_{K,n}}\rightarrow 0$ almost surely as $n\rightarrow\infty$.
\end{proof}

\begin{lemma}\label{l:large_k}
  Let
  \begin{equation*}
    \epsilon_K := \sum_{k=K+1}^\infty \left(\E I_{2,k}\right)^{1/2}.
  \end{equation*}
  Then 
  \begin{equation*}
     \norm{\overline{W}_{K,n}}_2 \leq \epsilon_K.
  \end{equation*}
\end{lemma}

\begin{Remark}
{\em
  A necessary and sufficient condition for $\epsilon_K \rightarrow 0$ as $K\rightarrow\infty$ 
  is~\eqref{eq:sigma_infty_assumption}.
Therefore, by Remark~\ref{r:tameness}, $\epsilon$-tameness for some $\epsilon < 1/2$ is sufficient. 
}
\end{Remark}

\begin{proof}[Proof of \refL{l:large_k}]
  Minkowski's inequality yields
  \begin{equation}
    \norm{\overline{W}_{K,n}}_2 \leq \frac{1}{\sqrt{n}}\sum_{k=K+1}^\infty \Big\|\sum_{\tau_k < i \leq n} (X_{k,i} - I_k)\Big\|_2.
    \label{eq:W_triangle}
  \end{equation}
  By conditioning on $C_k$, we can calculate the square of the $L^2$-norm 
  here:
  \begin{align}\label{eq:W_var_bound}
  \Big\|\sum_{\tau_k < i \leq n} (X_{k,i} - I_k)\Big\|_2^2 
  &= \E\,\E\Big[ \Big(\sum_{\tau_k < i \leq n}(X_{k,i} - I_k)\Big)^2\Big| C_k\Big] \nonumber \\
    &= \E\left\{ (n-\tau_k)^+ (I_{2,k} - I_k^2)\right\} \nonumber\\
    &\leq n \E I_{2,k},
  \end{align}
  where we use the fact that, conditionally given $C_k$, the random variables $X_{k,i}-I_k$ for $i > \tau_k$ are iid with zero mean. Substituting \eqref{eq:W_var_bound} into \eqref{eq:W_triangle} gives the result.
\end{proof}

\begin{proposition}\label{p:W_n_limit}
  Suppose that
  \begin{equation*}
    \sum_{k=1}^\infty \left(\E I_{2,k}\right)^{1/2} < \infty.
  \end{equation*}
  Then
  \begin{equation*}
    W_n\,\Lto\,\sigma_\infty Z,
  \end{equation*}
  where $Z$ has a standard normal distribution independent of $\sigma_\infty$.
\end{proposition}
\begin{proof}
  Let $t\in\R$ and $\delta> 0$. Since $W_n \leq t$ implies either 
	\[
	W_{K,n} \leq t + \delta \quad\text{ or }\quad \abs{W_n - W_{K,n}} > \delta,
	\]
	we have
  \begin{equation}
  \label{eq:breakup}
    \Prob[W_n\leq t] \leq \Prob[W_{K,n} \leq t + \delta] + \Prob[\abs{W_n - W_{K,n}} > \delta].
  \end{equation}
  Markov's inequality and \refL{l:large_k} imply
  \begin{equation}
  \label{eq:Markov}
    \Prob[\abs{W_n - W_{K,n}} > \delta] \leq \frac{\epsilon^2_K}{\delta^2}.
  \end{equation}
  Taking limits superior as $n\rightarrow\infty$ gives
\begin{align*}
\limsup_{n\to\infty} \Prob[W_n \leq t] 
&\leq \limsup_{n\to\infty} \Prob[W_{K, n} \leq  t + \delta] + \frac{\epsilon_K^2}{\delta^2} \\
&\leq \limsup_{n\to\infty} \Prob[T_{K, n} \leq  t + 2 \delta] + \frac{\epsilon_K^2}{\delta^2} \\
&= \Prob[\sigma_K Z \leq t + 2 \delta] + \frac{\epsilon_K^2}{\delta^2},
\end{align*}  
by~\eqref{eq:breakup}--\eqref{eq:Markov}, \refL{l:small_k}, and \refP{p:approx_limit}, respectively. 
Now taking limits as $K\rightarrow\infty$ gives
$$
    \limsup_{n\to\infty} \Prob[W_n \leq t] \leq \Prob[\sigma_\infty Z \leq t + 2 \delta]
$$
by \refL{l:sigma_limit} and the assumption that $\epsilon_K\to 0$. Letting $\delta\rightarrow 0$ yields
  \begin{equation}\label{eq:limsup}
    \limsup_{n\to\infty}\Prob[W_n \leq t] \leq \Prob[\sigma_\infty Z \leq t].
  \end{equation}
  Applying the previous argument with limsup replaced by liminf to
  \begin{equation*}
    \Prob[W_n \leq t] \geq \Prob[W_{K,n} \leq t - \delta] - \Prob[\abs{W_n - W_{K,n}} \geq \delta]
  \end{equation*}
  implies
  \begin{equation}
    \liminf_{n\to\infty}\Prob[W_n \leq t] \geq \Prob[\sigma_\infty Z < t].
    \label{eq:liminf}
  \end{equation}
  Since $\sigma_{\infty} Z$ has a continuous distribution, combining \eqref{eq:limsup} and \eqref{eq:liminf} gives the result.
\end{proof}

For completeness we include the following simple lemma, which will be needed in the sequel.
\begin{lemma}
  Let $0 < p < 1$ and $a_1, \ldots, a_n$ be nonnegative real numbers. Then
  \begin{equation*}
    \left(\sum_{k=1}^n a_k\right)^p \leq \sum_{k=1}^n a_k^p.
  \end{equation*}
  \label{l:calculus}
\end{lemma}
The final step in the proof of \refT{t:quickval_residual_limit} is to show that the difference between the centering random variable
\[
\sum_{k=1}^\infty (n-\tau_k)^+ I_k
\]
in $W_n$ and the more natural
\[
n S = \sum_{k=1}^\infty n I_k
\]
is negligible (when scaled by $1 / \sqrt{n}$) in the limit as $n\rightarrow\infty$. 
\begin{proposition}\label{p:n_tau_k_limit}
  If the source-specific cost function $\beta$ is $\epsilon$-tame with $\epsilon < 1/2$, then

  \begin{equation*}
    \frac{1}{\sqrt{n}}\sum_{k=1}^\infty\left[n-(n-\tau_k)^+\right]I_k \rightarrow 0
  \end{equation*}
  almost surely as $n\rightarrow\infty$.
\end{proposition}
\begin{proof}
  Observe that for any $0 < \delta < 1/2$, we have
  \begin{equation*}
    [n-(n-\tau_k)^+] = \min(n,\tau_k) \leq \tau_k^{(1/2) + \delta}n^{(1/2) - \delta}.
  \end{equation*}
  Therefore, if we let $0 < \delta < (1/2) - \epsilon$, it suffices to show that
  \begin{equation}\label{eq:tau_sum}
    \sum_{k=1}^\infty \tau_k^{(1/2)+\delta}I_k < \infty
  \end{equation}
  almost surely. We prove this by showing that the random variable in \eqref{eq:tau_sum} has finite expectation.
  Applying \cite[Lemma~3.2]{fn2013} implies that for the $\epsilon$-tameness constant $c$, we have
  \begin{equation*}
    I_k \leq \frac{2^\epsilon c}{1-\epsilon}(R_{k-1} - L_{k-1})^{1-\epsilon}.
  \end{equation*}
  Define, for $k=1,2,\ldots$, the sigma-field $\mathcal{F}_k:= \sigma\langle(L_1,R_1),\ldots (L_{k-1}, R_{k-1})\rangle$.
  Conditionally given $\mathcal{F}_k$, the distribution of $\tau_k$ is the convolution over $j=0,\ldots, k-1$ of geometric distributions with success probabilities $R_j - L_j$. 
  This distribution is stochastically smaller than the convolution of $k$ geometric distributions with success probability $R_{k-1}-L_{k-1}$. 
  Let $G_k, G_{k,0}, \ldots, G_{k,k-1}$ be $k + 1$ iid geometric random variables with success probability $R_{k-1}-L_{k-1}$. 
  Then
  \begin{align}
    \E\left[\left. \tau_k^{(1/2)+\delta}I_k \right| \mathcal{F}_k\right] & \leq
    C_1 \E\left[\left.\left(\sum_{i=0}^{k-1}G_{k,i}\right)^{(1/2)+\delta}(R_{k-1}-L_{k-1})^{1-\epsilon}\right|L_{k-1},R_{k-1}\right] \nonumber \\
    &\leq C_1 (R_{k-1}-L_{k-1})^{1-\epsilon}\,\E\left[\left. \sum_{i=0}^{k-1} G_{k,i}^{(1/2) + \delta}\right|L_{k-1},R_{k-1}\right] \nonumber \\
    &\leq C_1 k(R_{k-1}-L_{k-1})^{1-\epsilon}\,\E\left[\left. G_k^{(1/2) + \delta}\right|L_{k-1},R_{k-1}\right], \label{eq:conditional_tau_I}
 \end{align}
 where
 \begin{equation*}
   C_1 := \frac{2^\epsilon c }{1-\epsilon}.
 \end{equation*}
We can now compute
 \begin{equation}\label{eq:polylog}
   \E\left[ \left. G_k^{(1/2) + \delta} \right| L_{k-1}, R_{k-1}\right] = \sum_{i=1}^\infty z^{i-1}(1-z)i^p,
 \end{equation}
 where $z=1-(R_{k-1}-L_{k-1})\in[0,1)$ for $k\geq 2$ is the failure probability and $p=(1/2) + \delta$. Note that the infinite series in \eqref{eq:polylog} can be written in terms of a polylogarithm function, as follows:
   \begin{equation*}
     \sum_{i=1}^\infty z^{i-1}(1-z)i^p = z^{-1}(1-z)\Li_{-p,0}(z), \quad\quad
     \Li_{\alpha, r}(z): = \sum_{n=1}^\infty (\log i)^r \frac{z^i}{i^\alpha}.
   \end{equation*}
   Therefore \cite[Theorem~1]{f1999} implies the existence of an 
   $\eta\in(0,1)$ such that for $1-\eta < z < 1$, we have
   \begin{equation*}
     \sum_{i=1}^\infty z^i i^p \leq \Gamma(1+p)(1-z)^{-(1+p)}.
   \end{equation*}
   On $0 \leq z \leq 1-\eta$, the polylogarithm $\Li_{-p, 0}(z)$ is increasing and therefore we have the bound
   \begin{equation*}
    \Li_{-p,0}(z) \leq \sum_{i=1}^\infty (1-\eta)^i i^p =: C_{p,\eta}
   \end{equation*}
   Defining
   \begin{equation*}
     C_2 := \max(\Gamma(1+p), C_{p,\eta}),
   \end{equation*}
   for $z\in[0,1)$ we get
     \begin{equation}\label{eq:polylog_bound}
       \Li_{-p,0}(z) \leq C_2 (1-z)^{-(1+p)}.
     \end{equation}
     Substituting the bound from \eqref{eq:polylog_bound} in \eqref{eq:polylog} gives
     \begin{align*}
       \E\left[ G_k^p|L_{k-1}, R_{k-1}\right] &\leq C_2\frac{R_{k-1} - L_{k-1}}{1-(R_{k-1} - L_{k-1})} (R_{k-1} - L_{k-1})^{-(1+p)} \\
       &= C_2\sum_{j=0}^\infty (R_{k-1} -L_{k-1})^{j-p}.
     \end{align*}
    Therefore, after substituting $p=(1/2) + \delta$, an application of the monotone convergence theorem yields
    \begin{equation*}
      \E(\tau_k^{(1/2) + \delta}I_k) \leq C_3k \sum_{j=0}^\infty \E(R_{k-1} - L_{k-1})^{j+ (1/2) - \epsilon - \delta},
    \end{equation*}
    where $C_3:=C_1C_2$.
    Let $q:= (1 / 2) - \epsilon - \delta$; then by the restriction placed on $\delta$, we know $q > 0$. By
    \cite[Lemma~3.1]{fn2013}, we have
    \begin{equation*}
      \E(\tau_k^{(1/2) + \delta}I_k) \leq C_3 k \sum_{j=0}^\infty \left(\frac{2-2^{-(j+q)}}{j + q + 1}\right)^{k-1}.
    \end{equation*}
Therefore, after defining
\begin{equation*}
  \gamma_j := \frac{2-2^{-(j+q)}}{j + q + 1},
\end{equation*}
we have
\[
  \sum_{k=3}^\infty \E(\tau_k^{(1/2) + \delta}I_k) \leq C_3 \sum_{k=3}^\infty k \sum_{j=0}^\infty \gamma_j^{k-1}
  =C_3 \sum_{j=0}^\infty \sum_{k=3}^\infty k \gamma_j^{k-1}
  \leq 3C_3 \sum_{j=0}^\infty \frac{\gamma_j^2}{(1-\gamma_j)^2}.
\]
Consequently, to check the convergence in \eqref{eq:tau_sum}, it suffices to check that 
$
\sum_{j=0}^\infty \gamma_j^2 < \infty
$;
however, this follows trivially from the observation that
$
\gamma_j^2 \leq 4 / j^2
$.
Therefore, it remains to show that the $k = 1$ and $k = 2$ terms in \eqref{eq:tau_sum} have finite expectation. The first arrival time $\tau_1$ equals~$1$ identically and $\E I_1 < \infty$. Applying \eqref{eq:conditional_tau_I} when $k=2$ gives
\begin{equation*}\label{eq:tau_k2}
\E\left[ \left. \tau_2^{(1/2)+\delta}I_2 \right| \mathcal{F}_2\right] \leq 2C_1(R_1-L_1)^{1-\epsilon}\,
\E\left[\left.G_2^{(1/2)+\delta}\right|L_1,R_1\right].
\end{equation*}
Since $(R_1-L_1)^{1-\epsilon} < 1$ a.s.\ , it suffices to show that
\begin{equation}\label{eq:G2}
\E G_2^p < \infty
\end{equation}
for $p=(1/2) + \delta$. However, we can calculate the expectation in \eqref{eq:G2} exactly. Since $R_1-L_1 \Leq 1-U$, where $U$ has a unif$(0,1)$ distribution, 
\[
\E G_2^p = \sum_{i=1}^\infty i^p \E[(1-U)U^{i-1}]\\
  = \sum_{i=1}^\infty \frac{i^p}{i(i+1)},
\]
which is finite because $p<1$.
\end{proof}

\section{Convergence of moments for \quickval\ residual}\label{s:quickval_moments}

The main result of this section is that, under suitable tameness assumptions for the cost function, the moments of the normalized \quickval\ residual converge to those of its limiting distribution.

\begin{theorem}\label{t:quickval_residual_moments}
Let $p \in [2, \infty)$.
Suppose that the cost function $\beta$ is $\epsilon$-tame with $\epsilon < 1 / p$.  
Then the moments of orders $\leq p$ for the normalized \quickval\ residual 
  \begin{equation*}
\sqrt{n} \left( \frac{S_n}{n} - S \right)
  \end{equation*}
converge to the corresponding moments of the limit-law random variable $\sigma_\infty Z$.
  \label{t:QRM}
\end{theorem}

\begin{Remark}
\label{r:QRM}
\emph{
We will prove \refT{t:QRM} using the second assertion in \cite[Theorem 4.5.2]{c2001}.  Use of the first assertion in that theorem shows that, for all real~$r \in [1, p]$, we also have convergence of $r$th absolute moments.
}
\end{Remark}

As mentioned in Remark~\ref{r:QRM}, we prove \refT{t:QRM} using \cite[Theorem 4.5.2]{c2001} by proving that, for some $q > p$, the $L^q$-norms of the normalized \quickval\ residuals are bounded 
in~$n$.  Choosing $q$ arbitrarily from the nonempty interval $[2, 1 / \epsilon)$ and using the triangle inequality for $L^q$-norm, we do this by showing (in Lemmas~\ref{l:QRMW} and~\ref{l:QRMWhat}, respectively) that the same $L^q$-boundedness holds for each of the following two sequences:
\begin{align*}
  W_n 
&= \frac{1}{\sqrt{n}} \left[ S_n - \sum_{k=1}^\infty (n-\tau_k)^+ I_k \right]  = \frac{1}{\sqrt{n}} \sum_{k = 1}^{\infty} \left[ S_{k, n} - (n-\tau_k)^+ I_k \right] \\
  &= \frac{1}{\sqrt{n}} \sum_{k = 1}^{\infty} \sum_{\tau_k < i \leq n} (X_{k,i} - I_k),
\end{align*}
and the sequence previously treated in \refP{p:n_tau_k_limit}:
\[ 
\widehat{W}_n := \frac{1}{\sqrt{n}}\sum_{k=1}^\infty\left[n-(n-\tau_k)^+\right]I_k.
\]

\begin{lemma}
\label{l:QRMW}
Let $q \in [2, \infty)$, and suppose that the cost function~$\beta$ is $\epsilon$-tame with 
$0 \leq \epsilon < 1 / q$.  Then the sequence $(W_n)$ is $L^q$-bounded.
\end{lemma}

\begin{proof}
This is straightforward.  We proceed as at~\eqref{eq:W_triangle}, except that we use triangle inequality for $L^q$-norm rather than for $L^2$-norm:
  \begin{equation*}
    \norm{W_n}_q \leq \frac{1}{\sqrt{n}} \sum_{k = 1}^\infty 
    \left\| \sum_{\tau_k < i \leq n} (X_{k,i} - I_k) \right\|_q.
  \end{equation*}
  To bound the $L^q$-norm on the right, we employ Rosenthal's inequality \cite{r1970} conditionally given $C_k$ to find
\begin{align*}
\left\| \sum_{\tau_k < i \leq n} (X_{k,i} - I_k) \right\|_q^q 
&\leq c_q \left[ (n - \tau_k)^+ \| X_k - I_k \|_q^q 
  + \left( (n - \tau_k)^+ \right)^{q / 2} \| X_k - I_k \|_2^2 \right] \\
&\leq c_q \left[ n \| X_k - I_k \|_q^q + n^{q / 2} \| X_k - I_k \|_2^2 \right],
\end{align*}
and so, by \refL{l:calculus},
\[
\left\| \sum_{\tau_k < i \leq n} (X_{k,i} - I_k) \right\|_q 
\leq c_q^{1 / q} \left[ n^{1 / q} \| X_k - I_k \|_q + n^{1 / 2} \| X_k - I_k \|_2^{2 / q} \right].
\]
But by the argument at \eqref{eq:W_var_bound} we have
\[
\| X_k - I_k \|_2^2 \leq \E I_{2, k},
\]
and
\[
\| X_k - I_k \|_q \leq \|X_k\|_q + \|I_k\|_q = \left( \E I_{q, k} \right)^{1 / q} + \|I_k\|_q
\] 
by again conditioning on $C_k$ to obtain the equality here. 
Consider a generalization of the definition of 
$I_{2, k} = I_{2, k, k}$ given in \eqref{eq:I_2kl}:
\[
I_{q, k} := \E \left[ \left. X_k^q \right| C_k \right] = \int_{L_{k - 1}}^{R_{k - 1}}\!\beta^q(u, U_{\tau_k})\,du.
\]
Therefore
\[
 \left\| \sum_{\tau_k < i \leq n} (X_{k,i} - I_k) \right\|_q \leq c_q^{1 / q} \left\{ n^{1 / q} \left[ \left( \E I_{q, k} \right)^{1 / q} + \|I_k\|_q \right] + n^{1 / 2} \left( \E I_{2, k} \right)^{1 / q} \right\}.
 \]
Three applications of \refL{Ink} (requiring $\epsilon < 1 / q$, $\epsilon < 1$, and $\epsilon < 1 / 2$ to handle $\E I_{q, k}$, 
$\|I_k\|_q$, and $\E I_{2, k}$, respectively) do the rest.
\end{proof}

\begin{lemma}
\label{l:QRMWhat}
Suppose that the cost function~$\beta$ is $\epsilon$-tame with $0 \leq \epsilon < 1 / 2$.
Then the sequence $(\widehat{W}_n)$ is $L^q$-bounded for every $q < \infty$.
\end{lemma}

\begin{proof}
We may and do suppose $q \geq 2$.
We begin as in the proof of \refP{p:n_tau_k_limit}, except that there is now no harm in choosing 
$\delta = 0$.  So it is sufficient to prove that
\[
\sum_{k = 1}^{\infty} \left\| \tau_k^{1 / 2} I_k \right\|_q < \infty.
\]
We follow the proof of \refP{p:n_tau_k_limit} to a large extent; in particular, what we will show is that all the terms in this sum are finite and that, for sufficiently large~$K$, the series $\sum_{k = K}^{\infty}$ converges.  As in the proof of \refP{p:n_tau_k_limit}, we utilize the bound
  \begin{equation*}
    I_k \leq \frac{2^\epsilon c}{1-\epsilon}(R_{k-1} - L_{k-1})^{1-\epsilon},
  \end{equation*}
which requires only $\epsilon$-tameness with $\epsilon < 1$.  Then we proceed much the same way as at~\eqref{eq:conditional_tau_I}, but now substituting convexity of $q$th power for use of \refL{l:calculus}:
  \begin{align}
    \E\left[\left. \left( \tau_k^{1/2}I_k \right)^q \right| \mathcal{F}_k\right] 
    & \leq C_1^q\,\E \left[\left.\left(\sum_{i=0}^{k-1}G_{k,i}\right)^{q/2}(R_{k-1}-L_{k-1})^{q(1-\epsilon)}\right|L_{k-1},R_{k-1}\right] \nonumber \\
    &\leq C_1^q\,(R_{k-1}-L_{k-1})^{q(1-\epsilon)}k^{(q/2)-1}\,\E \left[ \left. \sum_{i=0}^{k-1} G_{k,i}^{q/2}\right|L_{k-1},R_{k-1}\right] \nonumber \\
    &\leq C_1^q k^{q/2} (R_{k-1}-L_{k-1})^{q(1-\epsilon)}\,\E\Big[ G_k^{q/2}\Big|L_{k-1},R_{k-1}\Big], \label{eq:conditional_tau_I_q}
 \end{align}
where, as before, $C_1 = 2^{\epsilon} c / (1 - \epsilon)$.

Arguing from here just as in the proof of \refP{p:n_tau_k_limit}, we find
\[
\E\left[ \left. G_k^{q/2} \right| L_{k-1}, R_{k-1}\right] \leq C_2\sum_{j=0}^\infty (R_{k-1} -L_{k-1})^{j-(q/2)}
\]
where $C_2 := \max(\Gamma(1+(q/2)), C_{q/2,\eta})$.  (See the proof of \refP{p:n_tau_k_limit} for the definition of $C_{q/2, \eta}$.)
Therefore, with $C_3 := C_1^{q/2} C_2$, we have
\[
\E\left[\left. \left( \tau_k^{1/2}I_k \right)^q \right| \mathcal{F}_k\right] 
\leq C_3\,k^{q / 2} \sum_{j = 0}^{\infty} (R_{k - 1} - L_{k - 1})^{j + q(1 - \epsilon) - (q/2)}.
\]
By \cite[Lemma~3.1]{fn2013}, we have (using our assumption $\epsilon < 1/2$ for the $j = 0$ term)
\begin{equation}
\label{eq:tau_I_bound_q}
\E\left( \tau_k^{1/2}I_k \right)^q
\leq C_3\,k^{q / 2} \sum_{j = 0}^{\infty} \gamma_{j, q, \epsilon}^{k - 1},
\end{equation}
where
\[
\gamma_{j, q, \epsilon} := \frac{2 - 2^{- [j + q (1 - \epsilon) - (q / 2)]}}{j + q (1 - \epsilon) - (q/ 2) + 1} 
\in (0, 1)
\]
decreases in~$j$ and vanishes in the limit as $j \to \infty$.
Therefore, taking $q$th roots and using \refL{l:calculus},
\[
\left\| \tau_k^{1 / 2} I_k \right\|_q \leq C_3^{q / 2} k^{1/2} 
\sum_{j = 0}^{\infty} \gamma_{j, q, \epsilon}^{(k - 1) / q}. 
\]
If we bound the factor $k^{1/2}$ here by $k$ and then sum the right side over $k \geq K$, the result is
\[
C_3^{q / 2}
\sum_{j = 0}^{\infty} 
\left[ (K - 1) \frac{\Gamma_j^{K - 1}}{1 - \Gamma_j} + \frac{\Gamma_j^{K - 1}}{(1 - \Gamma_j)^2} \right]
\leq
C_3^{q / 2}
K \sum_{j = 0}^{\infty} \frac{\Gamma_j^{K - 1}}{(1 - \Gamma_j)^2},
\]
where
\[
\Gamma_j \equiv \Gamma_{j, q, \epsilon} := \gamma_{j, q, \epsilon}^{1 / q} \in (0, 1),
\]
like $\gamma_{j, q, \epsilon}$, decreases in~$j$ and vanishes in the limit as $j \to \infty$.  Since
$\Gamma_j < (2 / j)^{1 / q}$, it follows if we take $K \geq 2 q + 1$ that
\[
\sum_{k = K}^{\infty} \left\| \tau_k^{1 / 2} I_k \right\|_q < \infty.
\]

It remains to show that $\left\| \tau_k^{1 / 2} I_k \right\|_q < \infty$ for every~$k$.  For this we 
use~\eqref{eq:tau_I_bound_q} to note, since $0 < \gamma_{j, q, \epsilon} < 2 / j$, that it clearly suffices to consider the cases $k = 1$ and $k = 2$.  When $k = 1$ we have $\tau_1$ = 1 and hence
$\left\| \tau_1^{1 / 2} I_1 \right\|_q = \left\| I_1 \right\|_q \leq C_1 < \infty$.  
Applying~\eqref{eq:conditional_tau_I_q} when $k = 2$ gives
\[
    \E\left[\left. \left( \tau_2^{1/2}I_2 \right)^q \right| \mathcal{F}_2\right] 
    \leq C_1^q 2^{q/2} (R_1 - L_1)^{q(1-\epsilon)}\,\E\Big[ G_2^{q/2}\Big|L_1, R_1 \Big], 
\]
and we can exactly compute
\begin{align*}
\lefteqn{\E \left\{  (R_1 - L_1)^{q(1-\epsilon)}\,\E\Big[ G_2^{q/2}\Big|L_1, R_1 \Big] \right\}} \\
&= \E \left\{  (R_1 - L_1)^{q(1-\epsilon)}\,
\sum_{i = 1}^{\infty} i^{q / 2} (R_1 - L_1) [1 - (R_1 - L_1)]^{i - 1} \right\} \\ 
&= \sum_{i = 1}^{\infty} i^{q / 2} \E \left[ U^{i - 1} (1 - U)^{q(1-\epsilon) + 1} \right]
= \sum_{i = 1}^{\infty} i^{q / 2} B(i, q(1 - \epsilon) + 2)
\end{align*}
where $U \sim \mbox{unif}(0, 1)$.  Each of the terms in this last sum is finite, and by Stirling's formula the $i$th term equals $(1 + o(1)) i^{-[2 + ((1/2) - \epsilon) q]} = o(i^{-2})$ as $i \to \infty$, so the sum converges.  Hence $\| \tau_2^{1/2} I_2 \|_q < \infty$.
\end{proof}

\begin{Remark}
\emph{
\cite[Chapter~7]{m2015} describes the approach, involving the contraction method for
inspiration and the method of moments for proof, we initially took in trying to establish
a limiting distribution for the \quickval\ residual in the special case of \quickmin\ with key-comparisons cost. 
It turns out that, for this approach, we must consider the
\quickmin\ limit and the residual from it bivariately.  However, we discovered
that, unfortunately, the limit residual \quickmin\ distribution
is not uniquely determined by its moments (we omit the proof here); so the method-of-moments approach
is ultimately unsuccessful, unlike for \quicksort~\cite{f2015}.  We nevertheless find that approach
instructive, since 
it does yield a rather direct proof of
convergence of moments for the residual in the special case of \quickmin with key-comparisons
cost; see \cite[Chapter~7]{m2015} for details.
}
\end{Remark}

\noindent
{\bf Acknowledgements.\ }The authors thank two anonymous referees for helpful comments.  In particular, one of the referees provided an argument enabling us to simplify and shorten the proof of \refL{L:EW_2^2} substantially.
\medskip

\noindent
{\bf Competing interests:\ }The authors declare none.
\medskip

\noindent
{\bf Data availability:\ }Data availability is not applicable to this article since no data were generated or analyzed.

\bibliographystyle{plain}
\bibliography{references}

\end{document}